\documentclass[reqno, 12pt]{amsart}
\usepackage{amsmath}
 \usepackage{amssymb}
\usepackage{amsthm}
\usepackage{times}
\usepackage{latexsym}
\usepackage[mathscr]{eucal}

\numberwithin{equation}{section}
 
  \newtheorem{theorem}{Theorem}[section]
  \newtheorem{proposition}[theorem]{Proposition}

  \newtheorem{example}[theorem]{Example}

\title[On a family of Einstein like Walker metrics]{On a family of Einstein like Walker metrics}

\author[Issa Allassane Kaboye, Mamadou Ciss, Abdoul Salam Diallo]{Issa Allassane Kaboye*, Mamadou Ciss**, Abdoul Salam Diallo***}

\newcommand{\acr}{\newline\indent}
\address{\llap{*\,} Universit\'e Andr\'e Salifou de Zinder,\acr
D\'epartement des Sciences Exactes,\acr
B. P. 656, Zinder, Niger}
\email{allassanekaboye@yahoo.fr}
\address{\llap{**\,} Universit\'e Alioune Diop,\acr
UFR SATIC, D\'epartement de Math\'ematiques,\acr
\'Equipe de Recherche en Analyse Non Lin\'eaire et G\'eom\'etrie (ANLG),\acr
 B. P. 30, Bambey, S\'en\'egal}
\email{mamadou.ciss@uadb.edu.sn}
\address{\llap{***\,} Universit\'e Alioune Diop,\acr
UFR SATIC, D\'epartement de Math\'ematiques,\acr
\'Equipe de Recherche en Analyse Non Lin\'eaire et G\'eom\'etrie (ANLG),\acr
B. P. 30, Bambey, S\'en\'egal }
\email{abdoulsalam.diallo@uadb.edu.sn}



\begin{document}

\begin{abstract}  
A four dimensional pseudo-Riemannian manifold of signature $(2, 2)$ 
is called a Walker manifold if it admits a parallel degenerate plane field. 
In this paper, we study the curvature properties of such a class of 
four dimensional Walker manifolds. In particular, we characterize Walker metrics of a given form which are Einstein-like.
\end{abstract}

\maketitle

\section{Introduction} 
\noindent
A pseudo-Riemannian metric $g$ on a four dimensional manifold $M$ is said to be 
a Walker metric if there exists a two dimensional null distribution on $M$, which
is parallel with respect to the Levi-Civita connection of $g$. This type of metrics 
has been introduced by Walker \cite{Walker1950} who has shown that they have 
a local canonical form depending on three smooth functions. Various curvature
properties of some special classes of Walker metrics have been studied 
in \cite{Brozos2009} where several examples of neutral metrics with interesting 
geometric properties have been given. Conditions for a restricted four dimensional 
Walker manifold to be Einstein, locally symmetric,  Einstein and locally conformally 
flat are given in \cite{Chaichi2005}.  Examples of  Walker Osserman metrics of signature 
$(3,3)$ which admits a field of parallel null $3$-planes are given in 
\cite{Diallo2011, Diallo2017}. A lot of examples of Walker structures 
have appeared, which proved to be important in differential geometry 
and general relativity as well .\\

\noindent
Two of the most extensively studied objects in Riemannian geometry 
and physics are Einstein manifolds and the Riemannian manifolds with constant scalar curvature. We denote by $\mathcal{E}$ the class of 
Einstein manifolds, by $\mathcal{C}$ the class of Riemannian manifolds with constant scalar curvature and by $\mathcal{P}$ the class of manifolds with parallel Ricci tensor. We have the following inclusion:
\begin{eqnarray*}
\mathcal{E} \subset \mathcal{P} \subset \mathcal{C}.
\end{eqnarray*}
Gray \cite{Gray1978} introduced two interesting classes of Riemannian manifolds which generalizes the concept of Einstein manifolds. The 
class of Riemannian manifolds admitting a cyclic parallel Ricci tensor denote by $\mathcal{A}$ and the class of Riemannian manifolds 
for which the Ricci tensor is a Codazzi tensor denote by $\mathcal{B}$. Note that, the classes of Riemannian manifolds $\mathcal{A}$
and $\mathcal{B}$ lie between the classes of Riemannian manifolds 
with parallel Ricci tensor $\mathcal{P}$ and the class of Riemannian manifolds with constant scalar curvature $\mathcal{C}$.\\

\noindent
In \cite{Batat2009}, the authors study the curvature properties of such 
a class of four-dimensional Walker manifolds. In particular, they characterize Walker metrics of a given form which are Einstein-like, conformally flat, locally symmetric. Motivated, by the following papers
\cite{Batat2009, Calvaruso2007}, 
we study a class of Walker metrics and we give condition for the Walker metric to be parallel Ricci tensor,  cyclic parallel Ricci tensor and the Ricci tensor is a Codazzi tensor. In this way, we organized 
the paper as follow:  in section \ref{Description}, we shall describe the curvature of the metric considered. In section \ref{Einsteinlike}, 
Einstein-like Walker metrics will be classified.

\section{Description of the metric}\label{Description}

\noindent
In this paper, we consider the family of metrics $g_{a}$ on
$O\subset \mathbb{R}^4$ given by
\begin{eqnarray}\label{eq21} 
g_{a} & = & 2(dx_1\circ dx_3 + dx_2 \circ dx_4 ) 
+ a(x_1,x_2,x_3,x_4)dx_3\circ dx_3\nonumber\\
&& + a(x_1,x_2,x_3,x_4)dx_4 \circ dx_4 ,
\end{eqnarray}
where $a$ is some fonction depending on $(x_1,x_2,x_3, x_4)$. We 
denote by 
$\partial_i := \frac{\partial}{\partial x_i}$ and
$a_i := \frac{\partial a(x_1, x_2, x_3, x_3)}{\partial x_i}$. 
From a straightforward calculation shows that the non-zero components  of the Levi-Civita connection of the metric (\ref{eq21}) are given by:
\begin{eqnarray}\label{eq22}
\nabla_{\partial_1}\partial_3 &=& \frac{1}{2}a_1\partial_1 , \; \nabla_{\partial_1}\partial_4 = \frac{1}{2}a_1\partial_2 , \;
\nabla_{\partial_2}\partial_3 = \frac{1}{2}a_2\partial_1 , \;\nabla_{\partial_2}\partial_4 = \frac{1}{2}a_2\partial_2 , \nonumber \\
\nabla_{\partial_3}\partial_3 &=& \frac{1}{2}(aa_1 + a_3)\partial_1
+ \frac{1}{2}(aa_2 - a_4)\partial_2 - \frac{1}{2}a_1\partial_3
- \frac{1}{2}a_2\partial_4 , \nonumber \\
\nabla_{\partial_3}\partial_4 &=& \frac{1}{2}a_4\partial_1
+ \frac{1}{2}a_3\partial_2 ,\nonumber\\
\nabla_{\partial_4}\partial_4 &=& \frac{1}{2}(aa_1 - a_3)\partial_1
+ \frac{1}{2}(aa_2 + a_4)\partial_2 - \frac{1}{2}a_1\partial_3
- \frac{1}{2}a_2\partial_4 .
\end{eqnarray}   

A curve $\gamma(t) = (x_1(t), x_2(t), x_3(t), x_4(t))$ in (\ref{eq21}) 
is a geodesic if and only if the following equations are satisfied:
\begin{eqnarray*}
0 &=& \ddot{x}_1 + a_1\dot{x}_1\dot{x}_3 + a_2\dot{x}_2\dot{x}_3
+\frac{1}{2} (aa_1 + a_3)\dot{x}_3\dot{x}_3 + a_4\dot{x}_3\dot{x}_4\\
&& + \frac{1}{2}(aa_1 - a_3)\dot{x}_4\dot{x}_4,\\
0 &=& \ddot{x}_2 + a_1\dot{x}_1\dot{x}_4 + a_2\dot{x}_2\dot{x}_4
+ \frac{1}{2}(aa_2  - a_4)\dot{x}_3\dot{x}_3 + a_3\dot{x}_3\dot{x}_4\\
&& + \frac{1}{2}(aa_2 + a_4)\dot{u}_4\dot{u}_4,\\
0 &=& \ddot{x}_3 - \frac{a_1}{2}\dot{x}_3\dot{x}_3  
- \frac{a_1}{2}\dot{x}_4\dot{x}_4 = 0,\quad
0 = \ddot{x}_4 - \frac{a_2}{2}\dot{x}_3\dot{x}_3 
- \frac{a_2}{2}\dot{x}_4\dot{x}_4.
\end{eqnarray*}
Recall that, a pseudo-Riemannian manifold $(M,g)$ is geodesically complete  if all geodesics exist for all time. The above PDE system is
hard to solve. So the geodesically completeness is not easy to prove.\\

Using (\ref{eq22}), we can completely determine the curvature tensor
of the metric (\ref{eq21}) by the following formula:
$\mathcal{R}(\partial_i, \partial_j)\partial_k = ([\nabla_{\partial_i}, 
\nabla_{\partial_j}]  - \nabla_{[\partial_i, \partial_i]}) \partial_k$. Then, 
taking into account (\ref{eq21}), we can determine all components 
of the $(0,4)$-curvature tensor 
$R_{ijkl}=g_{a}(\mathcal{R}(\partial_i,\partial_j)\partial_k, \partial_l)$. 
We obtain that, the non-zero component of the $(0,4)$-curvature tensor 
of the metric (\ref{eq21}) are given by:
\begin{eqnarray}\label{eq23}
R_{1313} &=& \frac{1}{2}a_{11}, \;
R_{1323} = \frac{1}{2}a_{12}, \;
R_{1424} = \frac{1}{2}a_{12} ,\;
R_{1334} = \frac{1}{4}(a_1a_2-2a_{14}) , \nonumber\\
R_{1414} & = & \frac{1}{2}a_{11} ,\;
R_{1434} = \frac{1}{4}(2a_{13} - a^{2}_1) ,\;
R_{2323} = \frac{1}{2}a_{22} , \nonumber\\
R_{2334} &=& \frac{1}{4}(a^{2}_2-2a_{24}) , \;
R_{2424} = \frac{1}{2}a_{22} ,\; 
R_{2434} = \frac{1}{4}(2a_{23} - a_1a_2) ,\nonumber\\
R_{3434} &=& \frac{1}{4}(2a_{33} + 2a_{44} - aa^{2}_1 - aa^{2}_2).
\end{eqnarray}
Using this, we can calculate the compoments $\rho_{ij}$ with respect 
to $\partial_i$ of the Ricci tensor of the metric (\ref{eq21}):  
\begin{eqnarray}\label{eq24}
\rho_{13} &=& \frac{1}{2}a_{11}, \;
\rho_{14} = \frac{1}{2}a_{12}, \;
\rho_{23} = \frac{1}{2}a_{12}, \;
\rho_{24} = \frac{1}{2}a_{22}, \nonumber\\
\rho_{33} &=& \frac{1}{2}(a^2_{2} + aa_{11} + aa_{22}
- 2a_{24}), \nonumber \\
\rho_{34} &=& \frac{1}{2}(-a_1a_2 + a_{14} + a_{23}),\nonumber\\
\rho_{44} &=& \frac{1}{2}(a^{2}_{1} + aa_{11} - 2a_{13} + aa_{22}).
\end{eqnarray}
The scalar curvature defined by $\tau = \mathrm{trace} \rho$ of the 
metric (\ref{eq21}) is:
\begin{eqnarray*}
\tau = a_{11} + a_{22}.
\end{eqnarray*}
We denote by $\mathcal{F}$ the Einstein tensor as 
$\mathcal{F}(X,Y) := \rho (X,Y) - \frac{\tau}{4} \cdot g(X,Y)$. We have:
\begin{eqnarray*}
\mathcal{F}_{13} &=& -\mathcal{F}_{24} =\frac{1}{4}(a_{11} - a_{22}),\;
\mathcal{F}_{14} = \mathcal{F}_{23} =\frac{1}{2}a_{12}, \\
\mathcal{F}_{33}&=& \frac{1}{4}(2a^{2}_{2} + aa_{11} + aa_{22} - 4a_{24}),\\
\mathcal{F}_{34} &=& \frac{1}{2}(- a_1a_2 + a_{14} + a_{23}),\\
\mathcal{F}_{44} &=& \frac{1}{4}(2a^{2}_{1} + aa_{11} - 4a_{13}
+ aa_{22}).
\end{eqnarray*}
The Walker metric (\ref{eq21}) is Einstein if $\mathcal{F}(X,Y)= 0$ for 
any vector fields $X,Y$. This is equivalent to:
\begin{eqnarray*} 
a_{11} -  a_{22} = 0,\; \; a_{12} = 0, \; \;
2a^{2}_{2} + aa_{11} + aa_{22} - 4a_{24}  = 0, \\
a_1a_2 - a_{14} - a_{23} = 0,\; \;
2a^{2}_{1} + aa_{11} + aa_{22} - 4a_{13}  =  0.
\end{eqnarray*}

\noindent
The non-zero compoments of the Ricci operator $Q$, given by $g(Q(X),Y)=\rho(X,Y)$ with respect to $\partial_i$ are: 
 \begin{eqnarray}\label{eq25}
Q_{11} &=& \rho_{13}, \; Q_{12} = \rho_{14},\; 
Q_{13} = \rho_{33} - a\rho_{13},\; Q_{14} = \rho_{34} - a\rho_{14},
\nonumber \\
Q_{21} & = & \rho_{14}, \; Q_{22} = \rho_{24}, \; 
Q_{23} = \rho_{34} - a\rho_{14}, \; Q_{24} = \rho_{44} - a\rho_{24}\nonumber \\
Q_{33} &=& \rho_{13}, \; Q_{34} = \rho_{14}, \; Q_{43} = \rho_{14}, \;
Q_{44} = \rho_{24}.
 \end{eqnarray}
Then, according to the above, it is easy to see that the eigenvalues of Ricci operator are solutions of
\begin{eqnarray*}
\left[\left(\rho_{13}-\lambda\right)\left(\rho_{24}-\lambda\right)-\rho^{2}_{14}\right]^2.
\end{eqnarray*}
If $\rho_{13}=\rho_{24}\;\mbox{and}\;\rho_{14}=0$, $\lambda=\rho_{13}=\frac{1}{2}a_{11}$ is the only Ricci eigenvalue. In this case, it is easy 
to see that the corresponding eigenspace is not four-dimensional (and 
so, $Q$ is not diagonalizable), unless $\rho_{13}-\rho_{24}=\rho_{14}=\rho_{33}-a\rho_{13}=\rho_{34}=\rho_{44}-a\rho_{13}=0$; that is,  
$a$ satifies
\begin{eqnarray}\label{eq26}
a_{11}-a_{22}&=& a_{12} = a^{2}_{2} + aa_{11} - 2a_{24} 
= a_{1}a_{2} - a_{14} - a_{23}
\cr&=&a^{2}_{1}+aa_{11}-2a_{13}=0.
\end{eqnarray}
If $\rho_{13}\neq\rho_{24}\;\mbox{or}\;\rho_{14}\neq0$, then $Q$ admits the eigenvalues
\begin{eqnarray*}\label{eq27}
\lambda_{\varepsilon}= \frac{\rho_{13}+\rho_{24}+ \varepsilon\sqrt{(\rho_{13}-\rho_{24})^{2}+\rho_{14}^2}}{2},
\end{eqnarray*}
where $\varepsilon=\pm 1$, each of multiplicity $2$. In this case, it is
easily seen by (\ref{eq21}) that $Q$ is not diagonalizable, unless
\begin{eqnarray}
 2\rho_{14}\rho_{34}-2a\rho_{14}^2- \rho_{13}\rho_{44}
 +  2a\rho_{13}\rho_{24}- \rho_{24}\rho_{33} &=& 0  \nonumber\\
 \rho_{44}-a\rho_{24} +  \rho_{33} - a \rho_{13} &=& 0, 
\end{eqnarray}
equivalently, equations above (\ref{eq27}) are equivalent to requiring that the defining function $a$ satisfies :
\begin{eqnarray}\label{eq28}
0 &=& a_{1}^{2} - 2a_{13} + aa_{22} + a_{2}^{2} + a a_{11} -2 a_{24}\nonumber \\
 &=&  a_{12} (-a_1a_2 + a_{14} + a_{23} - aa_{12}) \nonumber \\
&& - \frac{1}{2} a_{11} (a_1^2 + aa_{11}- 2 a_{13} ) 
 - \frac{1}{2}a_{22} ( a_{2}^2 + a a_{22} - 2 a_{24} ). 
\end{eqnarray}
Note that (\ref{eq26}) implies (\ref{eq27}). Hence, we can state the following result:

\begin{proposition}
The Walker metric  (\ref{eq21}) has a diagonalizable Ricci operator
only if its defining function $a$ satisfies (\ref{eq28}).
\end{proposition}

Note that, Ricci-parallel Riemannian manifolds have a diagonalizable 
Ricci curvature and are therefore isometric to a product of Einstein manifolds, at least locally. This is no longer true for pseudo-Riemannian manifolds \cite{Boubel2001}.\\ 

We can now calculate the covariant derivative of the Ricci tensor 
$\nabla\rho$ of the metric (\ref{eq21}).  By definition:
\begin{eqnarray*}
(\nabla_{\partial_i}\rho)_{jk} = \nabla_{\partial_i}\rho(\partial_j,\partial_k)
- \rho(\nabla_{\partial_i} \partial_j, \partial_k)
- \rho(\partial_j, \nabla_{\partial_i} \partial_k),
\end{eqnarray*}
By using (\ref{eq22}) and (\ref{eq24}), we prove the following:

\newpage
\begin{proposition}
The non-zero components of the covariant derivative $\nabla\rho$ 
of the Walker metric  (\ref{eq21})  are given by
\begin{eqnarray*}
(\nabla_{\partial_1}\rho)_{13}&=& \frac{1}{2}a_{111}, \;
(\nabla_{\partial_3}\rho)_{13} = \frac{1}{4}(2a_{113}+a_2a_{12}), \;
(\nabla_{\partial_1}\rho)_{14} = \frac{1}{2}a_{121},\\
(\nabla_{\partial_4}\rho)_{13} &=& \frac{1}{4}(2a_{114}-a_1a_{12}),\;
(\nabla_{\partial_3}\rho)_{14} = \frac{1}{4}(2a_{123}-a_1a_{12}),\\
(\nabla_{\partial_4}\rho)_{14} &=& \frac{1}{4}(2a_{124}-a_1a_{22}+a_1a_{11}+a_2a_{12}),\\
(\nabla_{\partial_3}\rho)_{23} &=& \frac{1}{4}(2a_{123}-a_2a_{11}+a_1a_{12}+a_2a_{22}),\\
(\nabla_{\partial_1}\rho)_{24} &=& \frac{1}{2}a_{122}, \;
(\nabla_{\partial_2}\rho)_{24} = \frac{1}{2}a_{222},\;
(\nabla_{\partial_3}\rho)_{24} = \frac{1}{4}(2a_{223} - a_2a_{12}),\\
(\nabla_{\partial_4}\rho)_{24} &=& \frac{1}{4}(2a_{224} + a_1a_{12}),
\; (\nabla_{\partial_4}\rho)_{23} = \frac{1}{4}(2a_{124}-a_2a_{12}),\\
(\nabla_{\partial_1}\rho)_{33} &=& \frac{1}{2}(2a_2a_{12}+aa_{111}+a_1a_{22}+aa_{221}-2a_{241}),\\
(\nabla_{\partial_2}\rho)_{33}&=& \frac{1}{2}(2a_2a_{22}+aa_{112}+a_2a_{22}+ aa_{222}-2a_{242}),\\
(\nabla_{\partial_3}\rho)_{33}&=& \frac{1}{2}(3a_2a_{23} + aa_{113} 
+ a_3 a_{22} + aa_{223} - 2a_{243} - aa_2a_{12}\\
&& + a_4a_{12} + aa_1a_{22} - 2a_1a_{24} + a_2a_{14}),\\
(\nabla_{\partial_4}\rho)_{33} &=& \frac{1}{2}(2a_2a_{24}+aa_{114}+a_4a_{22}+aa_{224}-2a_{244}-a_3a_{12}),\\
(\nabla_{\partial_1}\rho)_{34} &=& \frac{1}{2}(a_{141}+a_{231}
- 2a_1a_{12} - a_2a_{11}),\\
(\nabla_{\partial_2}\rho)_{34}&=& \frac{1}{2}(-2a_2a_{12}-a_1a_{22}+ a_{142}+a_{232}),\\
(\nabla_{\partial_3}\rho)_{34} &=& \frac{1}{4}(- 4a_2a_{13}
- a_1a_{23} + 2a_{143} + 2a_{233} - 2a_{3}a_{12} - aa_1a_{12}\\
&& + a_4a_{22} + a_1a_{14} + aa_2a_{11} - a_4a_{11}),\\
(\nabla_{\partial_4}\rho)_{34} &=& \frac{1}{4}( - a_2a_{14}  - 4a_1a_{24} 
+ 2a_{144} + 2a_{234} - 2a_4a_{12} - a_3a_{22} \\
&& +a_3a_{11} - aa_2a_{12} + aa_1a_{22}+a_2a_{23}),\\
(\nabla_{\partial_1}\rho)_{44} &=& \frac{1}{2}(3a_1a_{11}+aa_{111}-2a_{131}+aa_{122}),\\
(\nabla_{\partial_2}\rho)_{44} &=& \frac{1}{2}(2a_1a_{12} + a_2a_{11}
+ aa_{112} - 2a_{132} + aa_{222}),\\  
(\nabla_{\partial_3}\rho)_{44} &=& \frac{1}{2}( 2a_1a_{13} + a_3a_{11}
+ aa_{113} - 2a_{133} + aa_{223} - a_4a_{12}),\\
(\nabla_{\partial_4}\rho)_{44} &=& \frac{1}{2}(3a_1a_{14} + a_4a_{11}
+ aa_{114} - 2a_{134} + aa_{224} - aa_1a_{12} + a_3a_{12}\\
&& + a_1a_{23} + aa_2a_{11} - 2a_2a_{13})
\end{eqnarray*}
\end{proposition}

\section{Einstein like Walker metrics}\label{Einsteinlike}

Einstein-like metrics were introduced and first studied by Gray \cite{Gray1978} in the Riemannian framework as natural generalizations 
of Einstein metrics. Since they are defined through conditions 
on the Ricci tensor, their definition extends at once to the 
affine and pseudo-Riemannian setting \cite{Diallo2023}. \\

Next, we assume that, the defining function $a$ on the Walker metric described by (\ref{eq21}) as in the following form:
\begin{eqnarray}\label{eq31}
a(x_1,x_2,x_3,x_4)=x_1b(x_3,x_4) + x_2c(x_3,x_4) + d(x_3,x_4),
\end{eqnarray}
 where $ b, c, d$ are $\mathcal{C}^{\infty}$ real valued functions. 
From scalar curvature, the Walker metric (\ref{eq31}) has vanishing
scalar curvature.

\begin{proposition}
The Walker metric (\ref{eq21}) and (\ref{eq31}) is Einstein if and only 
if the functions $b, c$ satisfies :
\begin{eqnarray*}
b_3 = \frac{1}{2}b^2, \quad
c_4 = \frac{1}{2}c^2, \quad
b_4 + c_3 = bc..
\end{eqnarray*}
\end{proposition}

\begin{example}
We set: $ b=c= \frac{1}{-\frac{1}{2}u_3 - \frac{1}{2}u_4 + 2 }$. Then
for		 
\begin{eqnarray*}
a(x_1, x_2, ux_3, x_4) = \frac{x_1}{-\frac{1}{2}x_3 - \frac{1}{2}x_4 + 2 } 
+ \frac{x_2}{-\frac{1}{2}x_3 - \frac{1}{2}x_4 + 2 } + \xi(x_3, x_4),
\end{eqnarray*}
the Walker metric (\ref{eq21}) and (\ref{eq31}) is Einstein.
\end{example}
 
 \begin{theorem}
The Walker metric $g_a$ described by (\ref{eq21}) and (\ref{eq31})
is parallel Ricci tensor if and only the functions $b$ and $c$ satisfies
\begin{eqnarray*}
& b^2 = 2b_3+k(x_4),\quad \quad \quad 
c^2=2c_4+l(x_3), \\
& 3cc_3-2c_{34}-2bc_4+b_4c = 0, \quad \quad \quad  
4cb_3 + bc_{3} - 2b_{34} - 2c_{33} - bb_4 = 0, \\
& cb_4 + 4bc_4 - 2b_{44} - 2c_{34} - cc_3 = 0, \quad \quad \quad 
3bb_4 - 2b_{34} + bc_3 - 2cb_3 = 0.
\end{eqnarray*}
\end{theorem}

\begin{proof}
Recall that, a pseudo-Riemannian manifold $(M,g)$ is parallel Ricci tensor 
or belongs to class $\mathcal{P}$ if and only if its Ricci tensor $\rho$ is parallel, that is,
\begin{eqnarray}\label{eq32}
 (\nabla_X \rho)(Y,Z) = 0,
 \end{eqnarray}
for all vector fields $X,Y,Z$ tangent to $M$. More precisely, as concerns  Walker metrics (\ref{eq21}) and (\ref{eq31}), applying (\ref{eq32}), we 
have : $g_a\in \mathcal{P}$ if and only if:
 \begin{eqnarray*}
(\nabla_{\partial_3}\rho)_{33} &=& 0, \;
(\nabla_{\partial_3}\rho)_{34} = 0, \; 
(\nabla_{\partial_3}\rho)_{44} = 0 ,\\
(\nabla_{\partial_4}\rho)_{33} &=& 0, \; 
(\nabla_{\partial_4}\rho)_{34} = 0, \; 
(\nabla_{\partial_4}\rho)_{44} = 0.
 \end{eqnarray*}
The above system  is equivalent to:
\begin{eqnarray*}
3cc_3 - 2c_{34} - 2bc_4 + b_4c = 0, \;
4cb_3 + bc_{3} - 2b_{34} - 2c_{33} - bb_4 =0 , \\
bb_3 - b_{33} = 0,\; cc_4 - c_{44} = 0 , \\
cb_4 + 4bc_4 - 2b_{44} - 2c_{34} - cc_3 =0, \;
3bb_4 - 2b_{34} + bc_3 - 2cb_3 =0 .
\end{eqnarray*}
\end{proof}

\begin{theorem}
The Walker metric $g_a$ described by (\ref{eq21}) and (\ref{eq31}) is
cyclic parallel Ricci tensor or belongs to class $\mathcal{A}$ if and only 
if the functions $b$ and $c$ satisfies
\begin{eqnarray*}
3cc_3 - 2c_{34} - 2bc_4 + cb_4 = 0,\;
3bb_4 - 2b_{34} + bc_3 - 2cb_3=0.
\end{eqnarray*}
\end{theorem}

\begin{proof}
Recall tat, a pseudo-Riemannian manifold $(M,g)$ is cyclic parallel Ricci tensor or belongs to class $\mathcal{A}$ if and only if its Ricci tensor 
$\rho$ satisfy
\begin{eqnarray}\label{eq33}
(\nabla_X \rho)(Y,Z) + (\nabla_Y \rho)(X,Z)+(\nabla_Z \rho)(Y,X) = 0,
\end{eqnarray}
 for all vector fields $X, Y,Z$ tangent to $M$. The relation (\ref{eq33}) is equivalent to requiring that $\rho$ is a Killing tensor, that is,
\begin{eqnarray}\label{eq34}
(\nabla_X \rho)(X,X)=0.
\end{eqnarray}
More precisely, as concerns Walker metrics (\ref{eq21}) and
(\ref{eq31}), applying (\ref{eq34}), we have: $g_a\in \mathcal{A}$ if 
and only if:
\begin{eqnarray*}
(\nabla_{\partial_3}\rho)_{33} = 0 , \; 
(\nabla_{\partial_4}\rho)_{44} = 0,
  \end{eqnarray*}
 that is equivalent to 
 \begin{eqnarray*}
 3cc_3 - 2c_{34} - 2bc_4 + b_4c = 0,\;
3bb_4 - 2b_{34} + bc_3 - 2cb_3 = 0.
 \end{eqnarray*}
\end{proof}

\begin{theorem}
The Walker metric $g_a$ described by (\ref{eq21}) and (\ref{eq31})
is Ricci-Codazzi manifold or belongs to class $\mathcal{B}$ if and only 
if the functions $b$ and $c$ satisfies
\begin{eqnarray*}
4cb_3 + bc_3 - 2b_{34} - 2c_{33} - bb_4 + 2cc_4 - 2c_{44} &=& 0,\\
cb_4 + 4bc_4 - 2b_{44} - 2c_{34} - cc_3 + 2bb_3 - 2b_{33}=0.
\end{eqnarray*}
\end{theorem}

\begin{proof}
Recall that, a pseudo-Riemannian manifold $(M,g)$ is called Ricci-Codazzi manifold i belongs to class $\mathcal{B}$ if and only if its Ricci tensor 
is a Codazzi tensor, that is,
\begin{eqnarray}\label{eq35}
(\nabla_X \rho)(Y,Z) = (\nabla_Y \rho)(X,Z),
\end{eqnarray}
for all vector fields $X, Y,Z$ tangent to $M$. More precisely, as concerns Walker metrics (\ref{eq21}) and (\ref{eq31}), applying (\ref{eq34}), we 
have: $g_a\in \mathcal{B}$ if and only if:
 \begin{eqnarray*}
(\nabla_{\partial_3}\rho)_{34} = (\nabla_{\partial_4}\rho)_{33} , \; 
(\nabla_{\partial_3}\rho)_{44} = (\nabla_{\partial_4}\rho)_{34}, 
\end{eqnarray*}
that is equivalent to
\begin{eqnarray*}
4cb_3 + bc_{3} - 2b_{34} - 2c_{33} - bb_4 + 2cc_4 - 2c_{44} &=& 0, \\   
cb_4 + 4bc_4 - 2b_{44} - 2c_{34} - cc_3 + 2bb_3 - 2b_{33} &=& 0. 
\end{eqnarray*}
\end{proof}

With our metric (\ref{eq21}) and (\ref{eq31}), we have:
\begin{eqnarray*}
\mathcal{E} \subset \mathcal{P} = \mathcal{A} \cap \mathcal{B} \subset\mathcal{A} \cup \mathcal{B}\subset\mathcal{C}.
\end{eqnarray*}

\end{document}